\documentclass[12pt]{amsart}

\usepackage{amsfonts, amssymb, curves, epic, graphicx, amsmath,verbatim, amscd}

\newcommand{\R}{\mathbb{R}}
\newcommand{\N}{\mathbb{N}}
\newcommand{\Q}{\mathbb{Q}}
\newcommand{\Z}{\mathbb{Z}}

\newcommand{\larr}{\left( \begin{array}{c}}
\newcommand{\rarr}{\end{array} \right) }

\newcommand{\lsqarr}{\left[ \begin{array}{c}}
\newcommand{\rsqarr}{\end{array} \right]}

\newcommand{\inv}{\varprojlim}

\newtheorem{theorem}{Theorem}[section]
\newtheorem{Theorem}[theorem]{Theorem}

\newtheorem{corollary}[theorem]{Corollary}
\newtheorem{Remark}[theorem]{Remark}
\newtheorem{lemma}[theorem]{Lemma}

\begin{document}

\title{Pure discrete spectrum for a class of one-dimensional substitution tiling systems}
\author{Marcy Barge} 
\address{Montana State University, Bozeman, MT, 59717}
\email{barge@math.montana.edu}

\date{\today} 

\begin{abstract}
We prove that if a primitive and non-periodic substitution is injective on initial letters, constant on final letters, and has Pisot inflation, then the 
$\R$-action on the corresponding tiling space has pure discrete spectrum. As a consequence, all $\beta$-substitutions for $\beta$ a Pisot simple Parry number have tiling dynamical systems with pure discrete spectrum, as do the Pisot systems arising, for example, from substitutions associated with the Jacobi-Perron and Brun continued fraction algorithms.
\end{abstract}

\maketitle

\tableofcontents

\section{Introduction}\label{intro} 
Substitution dynamical systems arise in materials science (when analyzing spectral properties of atomic arrangements), in number theory (when investigating arithmetical properties of expansions in various bases), in hyperbolic dynamics (in the construction of Markov partitions for toral automorphisms), in computer science, word combinatorics, and, generally, when considering hierarchical, self-similar structures. A natural and effective approach to understanding such a system is to compare it with an action by translation on a compact abelian group: in the `best-behaved' case, when the system is measurably isomorphic with such an action, the system is said to have {\em pure discrete spectrum}. 
\\
There are two dynamical systems commonly associated with a substitution $\phi$: a subshift denoted by $(\Sigma_{\phi},\Z)$ called the {\em substitutive system}; and the {\em tiling dynamical system} $(\Omega_{\phi},\R)$, which is a particular suspension of the substitutive system. Pure discreteness of the substitutive system implies pure discreteness of the tiling dynamical system, but the converse does not hold without additional conditions (see \cite{CS} and also \cite{BBK}, Section 8, for a description of the substitutive system when the tiling system has pure discrete spectrum). In this article we deal exclusively with the tiling dynamical system - it has the important advantage of supporting a substitution induced homeomorphism $\Phi$ which interacts with the translation $\R$-action via $\Phi(T-t)=\Phi(T)-\lambda t$ for tilings $T\in\Omega_{\phi}$ and real numbers $t$. Here $\lambda>1$, the {\em inflation} associated with $\phi$, is the largest eigenvalue of the {\em abelianization}, or {\em substitution matrix},  $A$ of $\phi$ whose $ij$-th entry is the number of times the $i$-th letter appears in the image under $\phi$ of the $j$-th letter. In order for the tiling dynamical system $(\Omega_{\phi},\R)$ to have pure discrete spectrum it is necessary that the inflation $\lambda$ be a Pisot number (an algebraic integer, all of whose other algebraic conjugates lie strictly inside the unit circle) (see \cite{S1}). A substitution with Pisot inflation is called a {\em Pisot substitution} and a substitution is said to be {\em irreducible} if the characteristic polynomial of its substitution matrix is irreducible over $\Q$. We will denote the alphabet of a substitution by $\mathcal{A}$.
\\
There are various algorithms for determining whether or not a given substitution dynamical system has pure discrete spectrum (see, in particular, \cite{AL1,AL2}); these may involve a considerable amount of computation and it is desirable to have, instead, simple and easily checked conditions that guarantee pure discrete spectrum.
It is conjectured, for example, that the system associated with an irreducible Pisot substitution has pure discrete spectrum (this is known as the {\em Pisot substitution conjecture}, see \cite{ABBLS}), but this has only been verified for substitutions on two letters (\cite{HS}). It happens that many substitutions $\phi$ that arise in applications have the following special form: every letter of the alphabet $\mathcal{A}$ occurs as a first letter of a word $\phi(a)$ for some $a\in\mathcal{A}$ and all such words end in the same letter. It is the main theorem of this article that the tiling dynamical system of any (primitive) Pisot substitution with this property has pure discrete spectrum.
\\

Given a substitution $\phi$ on the finite alphabet $\mathcal{A}$, say $\phi(a)=a_1\cdots a_{t(a)}$ for $a\in\mathcal{A}$, we will say that $\phi$ is {\em injective on initial letters} if the map $a\mapsto a_1$ is injective, and we will say that $\phi$ is {\em constant on final letters} if $a_{t(a)}=b_{t(b)}$ for all $a,b\in\mathcal{A}$. Examples of families of substitutions that are both injective on initial letters and constant on final letters include (powers of) $\beta$-substitutions for $\beta$ a simple Parry number, and substitutions arising in connection with continued fraction algorithms (see Section \ref{applications}). The following theorem is the main result of this article.
\\
\\
{\bf Theorem} (Theorem \ref{main theorem} in text): If $\phi$ is a primitive Pisot substitution that is injective on initial letters and constant on final letters then the tiling dynamical system $(\Omega_{\phi},\R)$ has pure discrete spectrum.
\\
\\

In outline, the argument will be as follows:
\begin{enumerate}

\item Every tiling dynamical system $(\Omega,\R)$ has a {\em maximal equicontinuous factor} $(\Omega_{\max},\R)$ and $(\Omega,\R)$ has pure discrete spectrum if and only if the factor map $\pi_{max}:\Omega\to\Omega_{max}$ is a.e. one-to-one. \footnote{If $\Omega=\Omega_{\phi}$ with $\phi$ Pisot, then $(\Omega_{max},\R)$ is a Kronecker action on a torus or solenoid of dimension equal to the algebraic degree of the inflation $\lambda$ (\cite{BKw,BBK}).}
\item Given $T,T'\in\Omega_{\phi}$ we write $T\approx_s T'$ provided $\pi_{max}(T)=\pi_{max}(T')$ and, for a dense set of $t\in\R$, there is $n(t)\in\N$ so that so that $\Phi^{n(t)}(T-t)$ and $\Phi^{n(t)}(T'-t)$ have exactly the same tiles at the origin. The relation $\approx_s$ is a closed equivalence relation.
\item If the system $(\Omega_{\phi},\R)$ does not have pure discrete spectrum,
the quotient system $(\Omega_{\phi}/\approx_s,\R)$ is isomorphic with a tiling dynamical system $(\Omega_{\phi_s},\R)$ for a primitive, non-periodic Pisot substitution $\phi_s$ that is also injective on initial letters and constant on final letters. The relation $\approx_s$ is trivial on $\Omega_{\phi_s}$.
\item If $\psi$ is a primitive, non-periodic, Pisot substitution that is injective on initial letters and constant on final letters, then $\approx_s$ is nontrivial on $\Omega_{\psi}$.
\end{enumerate}
The proof now follows by contradiction: Were $(\Omega_{\phi},\R)$ not to have pure discrete spectrum, the relation $\approx_s$ would be trivial on $\Omega_{\phi_s}$. But item (4) would apply with $\psi=\phi_s$ to say that $\approx_s$ is nontrivial on $\Omega_{\phi_s}$.
\\
\\
Many of the ingredients of the proof have been developed elsewhere: (1) holds quite generally and is a consequence of the Halmos - von Neumann theory (see Chapter 3 of \cite{W}, for example); that $\Phi$ is a homeomorphism is a result of Moss\'{e} (\cite{M}, or, more generally, \cite{Sol}); much of (2) and (3) appears in \cite{B} in arbitrary dimension; and (4) is derived from the main idea of \cite{BD}. The details are presented in Section \ref{proof} and applications are given in Section \ref{applications}.

\section{Background and Notation}\label{background}
Given an alphabet $\mathcal{A}$, $\mathcal{A}^*$ will denote the set of all finite, nonempty, words in $\mathcal{A}$.
A substitution $\phi:\mathcal{A}=\{1,\ldots,d\}\to \mathcal{A}^*$, $d>1$, is {\em primitive} if some power of its abelianization $A$ is strictly positive. In this case $A$ has a unique (up to scale) positive left eigenvector $\omega=(\omega_1,\ldots,\omega_d)$. The corresponding Perron-Frobenius eigenvalue $\lambda$  is called the {\em inflation} of $\phi$. The {\em prototiles} for $\phi$ are the labeled intervals $\rho_i:=([0,\omega_i],i)$.
(For convenience, we will sometimes confuse $\rho_i$ with $[0,\omega_i]\times\{i\}$.) A {\em tile} is a translate of a prototile, $\rho_i-t:=([-t,\omega_i-t],i)$, with {\em support} $spt(\rho_i-t):=[-t,\omega_i-t]$ and {\em type} $i$. A {\em tiling} is a collection $T$ of tiles whose supports cover $\R$ with the property that any two distinct tiles in $T$ have supports with disjoint interiors. A {\em patch} $P$ is a finite subset of a tiling with {\em support} equal to the union of the supports of its constituent tiles. Given a tile $\tau=\rho_i-t$, let $\Phi(\tau)$ be the patch $$\Phi(\tau)=\{\tau_1,\ldots,\tau_k\},$$
with $$\tau_j:=\rho_{i_j}-\lambda t+\sum_{l=1}^{j-1}\omega_{i_l},$$
where $\phi(i)=i_1\cdots i_k$. Extend $\Phi$ to patches $P$ by $\Phi(P):=\cup_{\tau\in P}\Phi(\tau)$, and likewise to tilings. Notice that $spt(\Phi(P))=\lambda\cdot spt(P)$.
\\
Given a primitive substitution $\phi$ as above, there are $k\in\N$ and $a,b\in \mathcal{A}$ so that $\phi^k(a)=a\cdots$, $\phi^k(b)=\cdots b$, and $ba$ is in the {\em language}, $\mathcal{L}:=\{w: w \text{ is a factor of }\phi^n(i)\text{ for some }n\in\N \text{ and }i\in\mathcal{A}\}$, of $\phi$. Let $P=\{\rho_b-\omega_b,\rho_a\}$ and let $T=\cup_{m\in\N}\Phi^{mk}(P)$. Then $T$ is a tiling and $\Phi^k(T)=T$. The {\em tiling space} of $\phi$, or {\em hull} of $T$, is $$\Omega_{\phi}:=cl\{T-t:t\in\R\},$$
with the closure taken in the {\em local topology} in which two tilings are close if a small translate of one agrees with the other in a large neighborhood of the origin. This is a metric topology (we will denote the metric by $d$),
in which $\Omega_{\phi}$ is compact and connected. Furthermore, the space $\Omega_{\phi}$ does not depend on $T$ and the $\R$-action by translation, $(S,t)\mapsto S-t$, on $\Omega_{\phi}$ is minimal and uniquely ergodic (see, for example, \cite{AP}).
\\

 For $\phi$ that is primitive and non-periodic (that is, there are no translation periodic tilings in $\Omega_{\phi}$), the map $\Phi:\Omega_{\phi}\to \Omega_{\phi}$ is a homeomorphism
(\cite{M},\cite{Sol}) satisfying $\Phi(T-t)=\Phi(T)-\lambda t$ for all $T\in\Omega_{\phi},\, t\in\R$.
\\
Given $T\in\Omega_{\phi}$ and $R\ge0$, the {\em  $R$-patch of $T$ at 0} is defined as $$B_R[T]:=\{\tau\in T:spt(\tau)\cap \bar{B}_R(0)\ne\emptyset\}.$$
(Here $\bar{B}_R(0)$ is the closed ball of radius $R$ at the origin.) Thus $d(T,T')$ is small if there is a $t$ with $|t|$ small, and a large $R$, so that $B_R[T]=B_R[T'-t]$.
\\
The substitution $\phi$ is a {\em Pisot substitution} if its inflation $\lambda$ is a Pisot number. For such $\phi$,
the maximal equicontinuous factor $(\Omega_{\max},\R)$ of the tiling dynamical system $(\Omega_{\phi},\R)$, which is unique up to topological isomorphism, is a (non-trivial) torus or solenoid of dimension equal to the algebraic degree of $\lambda$, the factor map $\pi_{max}:\Omega_{\phi}\to\Omega_{max}$ is uniformly finite-to-one, measure preserving (with respect to the unique invariant measure on $\Omega_{\phi}$ and Haar measure on $\Omega_{max}$), and almost everywhere $r$-to-one for an $r<\infty$ called the {\em coincidence rank} of $\phi$ (\cite{BKw,BBK}). The system $(\Omega_{\phi},\R)$ has pure discrete spectrum if and only if the coincidence rank, $r$, equals 1.
Besides semi-conjugating $\R$-actions, $\pi_{max}$ also semi-conjugates $\Z$-actions: there is a hyperbolic and algebraic automorphism $\Phi_{max}:\Omega_{max}\to\Omega_{max}$ with $\pi_{max}\circ\Phi=\Phi_{max}\circ\pi_{max}$. 

Tilings $T$ and $T'$ are {\em proximal} if $inf_{t\in\R}d(T-t,T'-t)=0$ and {\em strongly proximal} if for all $R>0$ there is $t_R\in\R$ with $B_R[T-t_R]=B_R[T'-t_R]$. These two relations are equal for Pisot $\phi$ (this is a consequence of the `Meyer property' - see \cite{BK}) and clearly $\pi_{max}(T)=\pi_{max}(T')$ if $T$ and $T'$ are proximal.
The equicontinuous structure relation for Pisot $\phi$ is {\em strong regional proximality}, $\sim_{srp}$, defined by: $T\sim_{srp}T'$ if and only if for each $R>0$ there  are $S_R,S'_R\in\Omega_{\phi}$ and $t_R\in\R$ so that $B_R[T]=B_R[S],\, B_R[T']=B_R[S'_R]$, and $B_R[S_R-t_R]=B_R[S'_R-t_R]$. That is, $\pi_{max}(T)=\pi_{max}(T')$ if and only if $T\sim_{srp}T'$ (see \cite{BK}). 
\\

\section{Proof of the Main Theorem}\label{proof}

Suppose that $\phi$ is a (primitive and non-periodic) Pisot substitution on the alphabet $\mathcal{A}=\{1,\ldots,d\}$. Tilings $T,T'\in\Omega_{\phi}$ are {\em eventually coincident at $t\in\R$} if there is $n\in\N$ so that $B_0[\Phi^n(T-t)]=B_0[\Phi^n(T'-t)]$. Note that if $T,T'$ are eventually coincident at $t$, then there is $\epsilon>0$ so that $T,T'$ are eventually coincident at $t'$ for all $t'\in(t-\epsilon,t+\epsilon)$. We will say that $T$ and $T'$ are {\em densely eventually coincident} if $T,T'$ are eventually coincident at $t$ for a set of $t$ dense in $\R$ (hence for an open dense set of $t$).
Define $\approx_s$ on $\Omega_{\phi}$ by $T\approx_s T'$ if and only if
$T$ and $T'$ are strongly regionally proximal and densely eventually coincident. 
\\
The following theorem is a compilation of one-dimensional restrictions of several results of \cite{B}. We include a proof for completeness.

\begin{Theorem}\label{quotient of Omega} $\approx_s$ is a closed equivalence relation, invariant under translation and under $\Phi$, and $(\Omega_{\phi}/\approx_s,\R)$ is the maximal equicontinuous factor of $(\Omega_{\phi},\R)$ if and only if $(\Omega_{\phi},\R)$ has pure discrete spectrum. The maximal equicontinuous factor map
on $\Omega_{\phi}$ factors through the quotient map from $\Omega_{\phi}$ to $\Omega_{\phi}/\approx_s$ and that quotient map is a.e. one-to-one. In particular, if $(\Omega_{\phi},\R)$ does not have pure discrete spectrum, then neither does $(\Omega_{\phi}/\approx_s,\R)$.
\end{Theorem}

\begin{proof} As mentioned in the previous section, $T\sim_{srp}T'$ is equivalent to $\pi_{max}(T)=\pi_{max}(T')$, so $\sim_{srp}$ is a closed equivalence relation. That `$T$ is densely eventually coincident with $T'$' really means `$T$ is open-and-densely eventually coincident with $T'$' implies that dense eventual coincidence is an equivalence relation, so $\approx_s$ is an equivalence relation. To see that it's closed, suppose that $T_n\to T$, $T'_n\to T'$ and $T_n\approx_s T'_n$ for each $n$. Then $T\sim_{srp}T'$. Fix $R>0$ and let $t_n,t'_n\to 0$ be so that (for large $n$) $B_R[T]=B_R[T_n-t_n]$ and $B_R[T']=B_R[T'_n-t'_n]$. By Corollary 5.8 of  \cite{BK} there are, up to translation, only finitely many pairs of the form $(B_R[S],B_R[S'])$ with $S\sim_{srp}S'$. As $T_n\sim_{srp}T'_n$, it follows that $t_n=t'_n$ for all but finitely many $n$. Since, for $|t|<R$, eventual coincidence of $T_n$ and $T'_n$ at $t$ depends only on the $R$-patches of $T_n$ and $T'_n$ at the origin, and $T_n,T'_n$ are densely eventually coincident, we see that $T$ and $T'$ are eventually coincident at a set of $t$ dense in $B_R(0)$. Thus $T\approx_sT'$ and $\approx_s$ is a closed equivalence relation.
\\
The $\Phi$- and $\R$-invariance of $\approx_s$ is clear. Also, since the relation $\approx_s$ is contained in the relation $\sim_{srp}$, $\pi_{max}$ factors through $\Omega_{\phi}/\approx_s$.
\\
Suppose now that $(\Omega_{\phi},\R)$ has pure discrete spectrum. We will argue that $(\Omega_{\phi}/\approx_s,\R)$ is the maximal equicontinuous factor of $(\Omega_{\phi},\R)$ by showing that strong regional proximality implies dense eventual coincidence. Suppose that $T\sim_{srp}T'$ and suppose that there are $t_0\in\R$ and $\epsilon>0$ so that $T$ and $T'$ are not eventually coincident at any $t\in(t_0-\epsilon,t_0+\epsilon)$. Pick $n_i\to\infty$ with $\Phi^{n_i}(T-t_0)\to S\in\Omega_{\phi}$ and $\Phi^{n_i}(T'-t_0)\to S'\in\Omega_{\phi}$. Then $S\sim_{srp}S'$. Suppose that $B_0[S-t]=B_0[S'-t]$ for some $t\in\R$. There are then (for large $i$) $t_i\to t$ with $B_0[\Phi^{n_i}(T-t_0)-t_i]=B_0[S-t]=B_0[S'-t]=B_0[\Phi^{n_i}(T'-t_0)-t_i]$ (we have again used that there are, up to translation, only finitely many pairs $(B_0[V],B_0[V'])$ with $V\sim_{srp}V'$). Then $B_0[\Phi^{n_i}(T-t_0-(t_i/\lambda^{n_i}))]=B_0[\Phi^{n_i}(T'-t_0-(t_i/\lambda^{n_i}))]$ so that $T$ and $T'$ are eventually coincident at $t_0+(t_i/\lambda^{n_i})$. We have a contradiction when $i$ is large enough so that $t_i/\lambda^{n_i}<\epsilon$. Thus $B_0[S-t]\ne B_0[S'-t]$ for all $t\in\R$. As there are only finitely many pairs $(B_0[S-t],B_0[S'-t])$, up to translation, we see that $d(S-t,S'-t)$ is bounded away from 0. Then by minimality of the $\R$-action and closed-ness of $\sim_{srp}$, $\pi_{max}:\Omega_{\phi}\to\Omega_{max}$ is at least 2-to-1 everywhere. But $\pi_{max}$ is a.e. 1-to-1 since $(\Omega_{\phi},\R)$ has pure discrete spectrum. Hence $T$ must be densely eventually coincident with $T'$. This proves that if $(\Omega_{\phi},\R)$ has pure discrete spectrum then $\sim_{srp}$ is the same as $\approx_s$; that is, $(\Omega_{\phi}/\approx_s,\R)$ is the maximal equicontinuous factor.
\\
Conversely, suppose that $(\Omega_{\phi},\R)$ does not have pure discrete spectrum. Then the coincidence rank of $\phi$ is $r\ge2$ and there are $z\in \Omega_{max}$ and $T_1,\ldots,T_r\in \pi_{max}^{-1}(z)$ with $T_i$ and $T_j$ disjoint for $i\ne j$ (this is from \cite{BK} - see Theorem 4 of \cite{B}). By minimality (and the fact that, up to translation, there are only finitely many pairs $(B_0[T_i-t],B_0[T_j-t])$), this is true for all $z\in\Omega_{max}$. Such $T_i$ and $T_j$ are not proximal, and therefore $\Phi^n(T_i)$ and $\Phi^n(T_j)$ are not proximal for any $n\in\N$ (this is because $\Phi$ is a homeomorphism with $\Phi(S-t)=\Phi(S)-\lambda t$, so $\Phi^{n}$ preserves proximality for all $n\in\Z$). Then by Lemma 5.12 of \cite{BK}, $\Phi^n(T_i)$ and $\Phi^n(T_j)$ are disjoint for all $n\in\N$ and $i\ne j$. It follows that the factor map from $\Omega_{\phi}/\approx_s$ to $\Omega_{max}$ is at least $r$-to-1 everywhere.  Since $\pi_{max}$ is almost everywhere $r$-to-1, the quotient map from $\Omega_{\phi}$ to $\Omega_{\phi}/\approx_s$ must be almost a.e. 1-to-1 (and, furthermore, $(\Omega_{\phi}/\approx_s,\R)$ is not the maximal equicontinuous factor).
\end{proof}

Assume now that $(\Omega_{\phi},\R)$ does not have pure discrete spectrum. We will push the relation $\approx_s$ down to the (disjoint) union $\cup_{i=1}^d\rho_i$ of the prototiles for $\phi$. Let the relation $R$ on $\cup_{i=1}^d\rho_i$ be defined by $(x,i)R(y,j)$ if and only if there are $T,T'\in\Omega_{\phi}$ and $t,t'\in\R$ so that: $T\approx_s T'$, $t+\rho_i\in T$, $t'+\rho_j\in T'$, and $t+x=t'+y$.

\begin{lemma} \label{epsilon}: Suppose that $\phi$ is Pisot and is constant on final letters. There is then $\epsilon>0$ so that if $(x_n,i_n)R(x_{n+1},i_{n+1})$ for $n=1,\ldots N-1$ and  $P_1$ and $P_N$ are the patches  $P_1=\{\rho_1-\omega_1-x_1,\rho_{i_1}-x_1\}$ and $P_N=\{\rho_1-\omega_1-x_N,\rho_{i_N}-x_N\}$, then $[-\epsilon,0]\subset spt(P_1)\cap spt(P_N)$ and $P_1$ and $P_N$ are eventually coincident at $s$ for a dense set of  $s\in[-\epsilon,0]$.
\end{lemma}
\begin{proof} Suppose $l\in \mathcal{A}=\{1,\ldots,d\}$ is such that  $\phi(a)=\cdots l$ for all $a\in\mathcal{A}$. Then the tiles $\rho_i-\omega_i$ and $\rho_j-\omega_j$ are eventually coincident at $s$ for all $s\in (-\omega_l/\lambda,0)$ for each $i,j\in\mathcal{A}$. Let $\epsilon=\omega_l/\lambda$. The conclusion of the lemma is certainly true if $N=1$. Suppose it to be true for
some $N\ge1$ and suppose that $(x_n,i_n)R(x_{n+1},i_{n+1})$ for $n=1,\ldots N$. Then $P_1$ and $P_N$ are densely eventually coincident on $[-\epsilon,0]$ and there are $T,T'\in\Omega_{\phi}$ and $t,t'\in\R$ so that: $T\approx_s T'$, $t+\rho_{i_N}\in T$, $t'+\rho_{i_{N+1}}\in T'$, and $t+x_N=t'+x_{N+1}$. Let $\rho_i-\omega_i+t$ and $\rho_j-\omega_j+t'$ be the tiles of $T$ and $T'$ immediately to the left of $t+\rho_{i_N}$ and $t'+\rho_{i_{N+1}}$. Then the patches $P_N':=\{\rho_i-\omega_i+t, t+\rho_{i_N}\}-t-x_N=\{\rho_i-\omega_i-x_N, \rho_{i_N}-x_N\}$ and $P_{N+1}':=\{\rho_j-\omega_j+t', t'+\rho_{i_{N+1}}\}-t-x_N=\{\rho_j-\omega_j-x_{N+1}, \rho_{i_{N+1}}-x_{N+1}\}$ are eventually coincident at $s$ for a dense set of $s$ in the intersection of their supports (because $T\approx_s T'$). But also,
$P_N$ and $P_N'$ are eventually coincident at all $s\in (-\epsilon,0)\subset [-\epsilon-x_N, \omega_{i_N}-x_N]$ as are $P_{N+1}$ and $P_{N+1}'$ at all $s\in (-\epsilon,0)\subset [-\epsilon-x_{N+1},\omega_{i_{N+1}}-x_{N+1}]$. Thus, recalling that dense eventual coincidence implies open dense eventual coincidence, we have that $P_{N+1}$ and $P_1$ are eventually coincident at $s$ for a dense set of in $[-\epsilon,0]$, and the lemma is established by induction.
\end{proof}

Given a patch $Q$ with support $[a,b]$, let $\bar{Q}$ be the periodic tiling $\bar{Q}:=\cup_{k\in\Z}(Q+k(b-a))$.
The main technical tool that we use in this article is the following theorem (which has analogues in all dimensions - see \cite{BSW}, Theorem 3.1).

\begin{theorem}\label{main tool} Suppose that $\phi$ is a primitive Pisot substitution and suppose there is a patch $Q$ for $\phi$ with support $[a,b]$ and a set $W\subset\R$ with the properties: the subgroup of $\R/(b-a)\Z$ generated by $\{w+(b-a)\Z:w\in W\}$ is dense in $\R/(b-a)\Z$; and $\bar{Q}$ and $\bar{Q}-w$ are densely eventually coincident for all $w\in W$. Then $(\Omega_{\phi},\R)$ has pure discrete spectrum.
\end{theorem}
\begin{proof} Given tilings $T,T'$ made of tiles for $\phi$ (but not necessarily in $\Omega_{\phi}$) let us write $ T\sim T'$ provided $T$ and $T'$ are eventually coincident at 0. For $v\in\R$, let $U_v:=\{x\in\R: \bar{Q}-x\sim\bar{Q}-v-x\}$; $U_v$ is open and $U_w$ is open and dense in $\R$ for all $w\in W$. So if $w,w'\in W$ then $U_w\cap U_{w'}$ is also open and dense in $\R$. For $x\in U_w\cap U_{w'}$ we have $\bar{Q}-w-x\sim\bar{Q}-x\sim\bar{Q}-w'-x$ and hence $\bar{Q}-w$ and $\bar{Q}-w'$ are densely eventually coincident. This means that $\bar{Q}-(w-w')$ and $\bar{Q}$ are densely eventually coincident. Note also that $U_v=\R$ for $v\in (b-a)\Z$. Thus, replacing $W$ by the subgroup of $\R$ generated by $W\cup (b-a)\Z$, we may assume that $W$ is a dense subgroup of $\R$. 
\\
We claim that if $U_v$ is nonempty for some $v\in\R$, then $U_v$ is dense in $\R$. To see this, suppose that $x\in U_v$ and fix $w\in W$. There is then $x'\in U_v\cap U_w\cap(U_w-v)$, $x'$  arbitrarily close to $x$.
Then $\bar{Q}-v-x'\sim \bar{Q}-x'$, $\bar{Q}-w-x'\sim \bar{Q}-x'$, and $\bar{Q}-w-(x'+v)\sim \bar{Q}-(x'+v)$. It follows that $\bar{Q}-v-(x'+w)\sim \bar{Q}-(x'+w)$; i.e., $x'+w\in U_v$. Since $x'$ is as close to $x$ as we wish, this shows that $x+W\subset cl(U_v)$, and thus $U_v$ is dense in $\R$, as claimed.
\\
Suppose now that $(\Omega_{\phi},\R)$ does not have pure discrete spectrum. There are $k\in\N$ and $T,T'\in\Omega_{\phi}$ so that $\Phi^k(T)=T$, $\Phi^k(T')=T'$, $T\sim_{srp}T'$, and $T\cap T'=\emptyset$. There are then $S,S'\in\Omega_{\phi}$ and $t\in \R$ so that $B_0[S]=B_0[T]$, $B_0[S']=B_0[T']$, and $B_0[S-t]=B_0[S'-t]$. Let $P\subset S$ and $P'\subset S'$ be patches with supports containing $[-1, t+1]$. We may assume that $P-s$ and $P'-s'$ are sub patches of $Q$ for some $s,s'\in\R$ (otherwise, replace $Q$ by $\Phi^m(Q)$ for sufficiently large $m$, so that, by primitivity, translates of $P$ and $P'$ occur in $\Phi^m(Q)$, and replace $W$ by $\lambda^mW$). From $S\cap S'\ne\emptyset$ we have $(\bar{Q}-(s'-s))\cap\bar{Q}\ne\emptyset$; so $U_{s'-s}\ne\emptyset$. By the above claim, $U_{s'-s}$ is dense in $\R$. Since $B_0[T']-s'\subset P'-s'\subset \bar{Q}$ and $B_0[T]-s'\subset P-s'\subset \bar{Q}-(s'-s)$, there is $x\in spt(B_0[T']-s')\cap spt(B_0[T]-s')$ with $\bar{Q}$ and $\bar{Q}-(s'-s)$ eventually coincident at $x$. Then $T$ and $T'$ are eventually coincident at $s'+x$. In particular, there is $n\in\N$ with $\Phi^n(T)\cap \Phi^n(T')\ne\emptyset$. Then $\Phi^{nk}(T)\cap \Phi^{nk}(T')\ne\emptyset$. But $\Phi^{nk}(T)=T$ and $\Phi^{nk}(T')=T'$ and $T\cap T'=\emptyset$. Thus $(\Omega_{\phi},\R)$ must have pure discrete spectrum.
\end{proof}

\begin{corollary} \label{corollary} Suppose that $\phi$ is a primitive Pisot substitution and suppose there is a tile $\tau$ for $\phi$ and a sequence $t_n\to 0$, $t_n\ne 0$, so that for each $n$ the set of $x$ so that $\tau$ and $\tau-t_n$ are coincident at $x$ is dense in $spt(\tau)\cap spt(\tau-t_n)$. Then $(\Omega_{\phi},\R)$ has pure discrete spectrum.
\end{corollary}
\begin{proof} Let $a$ be the type of $\tau$ and let $w$ be a word in the language of $\phi$ of the form $w=bub$, $b\in\mathcal{A}$. There is $m\in\N$ so that $w$ occurs in $\phi^m(a)$ and then there is a sub patch $Q$ of $\Phi^m(\tau)$ that follows the pattern of the word $bu$. Let $N\in\N$ be such that $\lambda^m|t_n|<\omega_b$ for all $n\ge N$. The tilings $\bar{Q}$ and $\bar{Q}-\lambda^mt_n$ are densely eventually coincident for $n\ge N$. Theorem \ref{main tool} applies with $W=\{\lambda^m t_n:n\ge N\}$.
\end{proof}

The following is Lemma 10 of \cite{B} with slightly weaker hypotheses.

\begin{lemma} \label{finite} Suppose that $\phi$ is Pisot, primitive,  and constant on final letters and suppose that $(\Omega_{\phi},\R)$ does not have pure discrete spectrum. There is then $B<\infty$ so that if $(x_n,i_n)R(x_{n+1},i_{n+1})$ for $n=1,\ldots N-1$, then
$\sharp\{(x_n,i_n):n=1,\ldots,N\}\le B$.
\end{lemma}
\begin{proof} Let $\epsilon>0$ be as in Lemma \ref{epsilon}. If there is no such $B$ then, using Lemma \ref{epsilon}, there are $i,j\in\mathcal{A}$ and $y_n\to y\in\R$, with $y_n\ne y_m$ for $n\ne m$, so that the patches $P=\{\rho_1-\omega_1,\rho_i\}$ and $P_n=\{\rho_1-\omega_1-y_n, \rho_j-y_n\}$ are eventually coincident at an open dense set $U_n$ of points in an interval $I_n\subset spt(P)\cap spt( P_n)$ of length $\epsilon$. Let's pick a subsequence $\{n_k\}$ with $I_{n_k}\to I$ in the Hausdorff topology. Let $J$ be the middle third of $I$, let $J'$ be the middle third of $J$ and let $K$ be large enough so that $|y_{n_{k}}-y_{n_{k'}}|<\epsilon/9$ and $J\subset I_{n_k}$ for all $k,k'\ge K$. Then, with $x_m:=y_{n_{(m-1)K}}-y_{n_{mK}}$ for $m\in\N$,  $P$ and $P-x_m$ are eventually coincident on the dense subset $U':=J'\cap(\cap_{k,k'\ge K}(U_{n_k}-(y_{n_{k'}}-y_{n_k})))$ of $J'$. For large enough $s\in\N$ there is a tile $\tau\in\Phi^s(P)$ with support contained in $\lambda^s J'$. We have that $\tau$ and $\tau-\lambda^s x_m$ are eventually coincident at a dense set of points in the intersection of their supports for all $m\in\N$. Then $(\Omega_{\phi},\R)$ has pure discrete spectrum by Corollary \ref{corollary}.
\end{proof}

Define $\bar{R}$ on  $\cup_{i=1}^d\rho_i$ to be the transitive closure of $R$:  $(x,i)\bar{R}(y,j)$ if and only if there are 
$(x_n,i_n),\, n=1,\ldots, N$ with $(x_n,i_n)R(x_{n+1},i_{n+1})$ for $n=1,\ldots N-1$ and $(x,i)=(x_1,i_1),(y,j)=(x_N,i_N)$.

\begin{lemma} \label{closed} Suppose that $\phi$ is Pisot, primitive,  and constant on final letters and suppose that $(\Omega_{\phi},\R)$ does not have pure discrete spectrum. Then $\bar{R}$ is a closed equivalence relation.
\end{lemma}
\begin{proof} Clearly $\bar{R}$ is an equivalence relation. That $R$ is closed follows from $\approx_s$ being closed; by Lemma \ref{finite}, $\bar{R}$ is also closed.
\end{proof}

Under the hypotheses of the two previous lemmas, let $E_i:=\{(0,i),(\omega_i,i)\}$, $E:=\cup_{i=1}^dE_i$, $[E]:=\{(y,j):(y,j)\bar{R}(x,i) \text{ for some } (x,i)\in E_i,i\in\{1,\ldots,d\}\}$, and $[E]_i:=[E]\cap\rho_i$. The first coordinates of elements of $[E]_i$ partition the support of $\rho_i$ into finitely many subintervals $I_i^1\le I_i^2\le\cdots \le I_i^{m(i)}$.

\begin{lemma} \label{gluing} Suppose that $\phi$ is Pisot, primitive, and constant on final letters and suppose that $(\Omega_{\phi},\R)$ does not have pure discrete spectrum. If $x\in \mathring{I}_i^j$, $y\in \mathring{I}_k^n$, and $(x,i)\bar{R}(y,k)$, then
$I_i^j-x=I_k^n-y$ and $(x',i)\bar{R}(y',k)$ for all $x'\in\mathring{I}_i^j$ and $y'=x'-x+y\in \mathring{I}_k^n$.
\end{lemma}
\begin{proof} Let $[(x,i)]$ denote the $\bar{R}$ equivalence class of $(x,i)$ and suppose that $(x,i)\notin [E]_i$.
For $(y,k)\in [(x,i)]$, $y\in I_k^n$ for some unique $n\in\{1,\ldots,m(k)\}$: set $min(y,k):=min(I_k^n)$ and $max(y,k):=max(I_k^n)$.
Let $t^{min}_{(x,i)}:=max_{(y,k)\in[(x,i)]}(x-y+min(y,k))$ and $t^{max}_{(x,i)}:=min_{(y,k)\in[(x,i)]}(x-y+max(y,k))$. Then for $(y,k)\in[(x,i)]$ we have  $t^{min}_{(y,k)}=t^{min}_{(x,i)}+(y-x)$ and $t^{max}_{(y,k)}=t^{max}_{(x,i)}+(y-x)$.
It follows that for any $t$ with $t^{min}_{(x,i)}<x-t<t^{max}_{(x,i)}$ and any $(y,k)\in[(x,i)]$ we have $(x-t,i)\bar{R}(y-t,k)$ and  $t^{min}_{(y,k)}<y-t<t^{max}_{(y,k)}$. (This is clearly true if $(x,i)R(y,k)$; then apply this in finitely many steps.) If $x\in \mathring{I}_i^j$, $y\in \mathring{I}_k^n$ and $(x,i)R(y,j)$, then $I_i^j=[t^{min}_{(x,i)},t^{max}_{(x,i)}]$, $I_k^n=[t^{min}_{(y,j)},t^{max}_{(y,j)}]$, and we have: $I_i^j-x=I_k^n-y$ and $(x',i)\bar{R}(y',k)$ for all $x'\in\mathring{I}_i^j$ and $y'=x'-x+y\in \mathring{I}_k^n$.
\end{proof}

Let us say that intervals $I_i^j$ and $I_k^n$ as in the Lemma above are equivalent and let $\{J_1,\ldots,J_m\}$ be the distinct equivalence classes of the $I_i^j$. For each $k\in\{1,\ldots,m\}$, let $l_k$ be the common length of the $I_i^j\in J_k$.  We make new prototiles $\alpha_k:=([0,l_k],k)$ and a new substitution $\phi_s$ on 
$\{1,\ldots,m\}$ as follows. Let $X_{\phi}:=\cup_{i=1}^d\rho_i/E$ be the wedge of circles obtained by identifying the endpoints in the disjoint union of the old prototiles, let $f_{\phi}:X_{\phi}\to X_{\phi}$ be the map that locally stretches length by a factor of $\lambda$ and follows the pattern of $\phi$, and let $\pi:\cup_{i=1}^d\rho_i\to X_{\phi}$ be the quotient map. Note that if $\pi_1:\Omega_{\phi}\to X_{\phi}$ is given by $\pi_1(T)=\pi((x,i))$ provided $\rho_i-x\in T$ and $x\in[0,\omega_i]$, then $f_{\phi}\circ \pi_1=\pi_1\circ\Phi$. Since $\approx_s$ is invariant under $\Phi$, $\bar{R}$, pushed forward to $X_{\phi}$ (which we continue to call $\bar{R}$) is invariant under $f_{\phi}$. It follows from Lemma \ref{gluing} that $X_{\phi}/\bar{R}$ is also a wedge of circles, one for each $J_k$, and $f_{\phi}$ induces a continuous map $f_{\phi_s}$ on $X_{\phi_s}:=X_{\phi}/\bar{R}$ satisfying: If $\pi_s:X_{\phi}\to X_{\phi_s}$ is the quotient map, then $\pi_s\circ f_{\phi}=f_{\phi_s}\circ \pi_s$. Now let $\phi_s$ be the substitution on the alphabet $\{1,\ldots,m\}$ with the property that  $\phi_s(k)=k_1k_2\cdots k_r$ if $f_{\phi_s}$ maps $\pi_s(\alpha_k)$ first around $\pi_s(\alpha_{k_1})$, then around $\pi_s(\alpha_{k_2})$, ..., finally around $\pi_s(\alpha_{k_r})$. Let $\alpha:\{1,\ldots,d\}\to\{1,\ldots,m\}$ be the morphism $\alpha(i):= i_1\cdots i_{m(i)}$ with $i_j$ defined by $I_i^j\in J_{i_j}$; we have $\alpha\circ\phi=\phi_s\circ\alpha$.

\begin{lemma} \label{shape of phi sub s} Suppose that $\phi$ is primitive, Pisot, and constant on final letters and suppose that $(\Omega_{\phi},\R)$ does not have pure discrete spectrum. Then $\phi_s$ is primitive, Pisot, and constant on final letters and if $\phi$ is injective on initial letters, so is $\phi_s$.
\end{lemma}
\begin{proof}  If $J_k$ is an equivalence class of intervals, each $I_i^j\in J_k$ has the property that its endpoints $a<b$ satisfy $(a,i)$ and $(b,i)$ are each  $\bar{R}$-equivalent to an endpoint of some prototile: say
$(a,i)\bar{R}(0,r)$ and $(b,i)\bar{R}(\omega_t,t)$. Then $I_r^1, I_t^{m(t)}\in J_k$. Let $l\in\mathcal{A}$ be such that $\phi(c)=\cdots l$ for each $c\in\mathcal{A}$ and let $l'$ be such that $I_l^{m(l)}\in J_{l'}$. Then $$\phi_s\circ \alpha (t)=\phi_s(\cdots k)=\cdots\phi_s(k).$$ On the other hand, $$\alpha\circ \phi(t)=\alpha(\cdots l)=\cdots l'.$$ From $\phi_s\circ\alpha=\alpha\circ \phi$, we have $\phi_s(k)=\cdots l'$ for all $k$; that is, $\phi_s$ is constant on final letters. 

Primitivity of $\phi_s$ follows from primitivity of $\phi$, the relation $\alpha\circ\phi=\phi_s\circ\alpha$, and surjectivity of $\pi_s$.

If $\phi$ is injective on initial letters there is $n\in\N$ so that $\phi^n(c)=c\cdots$ for all $c\in\mathcal{A}$. Then  from $\alpha(k)=r\cdots$ we have $$\phi_s^n\circ\alpha(r)=\phi_s^n(k\cdots)=\phi_s^n(k)\cdots$$
and also $$\alpha\circ \phi^n(r)=\alpha(r\cdots)=k\cdots,$$
from which it follows that $\phi_s^n(k)=k\cdots$ for all $k$, and $\phi_s$ is injective on initial letters. 
\end{proof}

\begin{Remark} If $k$ is any letter in the alphabet for $\phi_s$, there are $r,t\in\mathcal{A}$ with $\alpha(r)=k\cdots$ and $\alpha(t)=\cdots k$ (these are the $r$ and $t$ of the first two sentences of the proof of Lemma \ref{shape of phi sub s}). In particular, the alphabet for $\phi_s$ is no larger than the alphabet for $\phi$. The construction above of $\phi_s$ depends on the finiteness of the equivalence classes of the relation $\bar{R}$, and to obtain this we have used the constancy of $\phi$ on final letters. It would be interesting to know if this condition can be dropped (or considerably weakened) in Lemma \ref{finite}. As long as $(\Omega_{\phi},\R)$ does not have pure discrete spectrum, $\Omega_{\phi}/\approx_s$ is (isomorphic to) a substitution tiling space for some substitution $\psi$ (one can see this by first collaring, or rewriting, $\phi$ so that some power of the collared, or rewritten, version is constant on final letters, and then applying the following theorem). What can one say about the possiblities for the shape of $\psi$ and the size of its alphabet?
\end{Remark}

\begin{Theorem} \label{structure} Suppose that $\phi$ is primitive, Pisot, and constant on final letters and suppose that $(\Omega_{\phi},\R)$ does not have pure discrete spectrum. Then the substitution $\phi_s$ is Pisot and the dynamical system $(\Omega_{\phi}/\approx_s,\R)$ is topologically isomorphic with $(\Omega_{\phi_s},\R)$ by an isomorphism that conjugates the substitution homeomorphisms. Furthermore, the relation $\approx_s$ is trivial on $\Omega_{\phi_s}$ and $(\Omega_{\phi_s},\R)$ does not have pure discrete spectrum.
\end{Theorem}
\begin{proof} This is basically the 1-dimensional version of Theorem 12 of \cite{B}. We provide an argument here for the reader's convenience. 

The morphism $\alpha$ induces a map $\bar{\alpha}:\Omega_{\phi}\to\Omega_{\phi_s}$ that semi conjugates $\Phi$ with $\Phi_s$ and also the $\R$-actions on the two spaces (see, for example, \cite{BD2}). We will show that $\bar{\alpha}(T)=\bar{\alpha}(T')$ if and only if $T\approx_s T'$. 

Let $\pi_1:\Omega_{\phi}\to X_{\phi}$ be given by $\pi_1(T)=[(x,i)]$ if and only if $\rho_i-x\in T$ and, likewise, let $\pi_2:\Omega_{\phi_s}\to X_{\phi_s}$ be defined by $\pi_2(S)=[(x,k)]$ if and only if $\alpha_k-x\in S$. We have the following commuting diagram. 
\begin{equation}\label{diagram}
\begin{picture}(200, 50)(-50, -10)

\put(27, 0){$\pi_1$}
\put(90, 0){$\pi_2$}
\put(50, 27){$\bar{\alpha}$}

\put(2,20){$\circlearrowleft$}
\put(2,-25){$\circlearrowleft$}
\put(100,20){$\circlearrowleft$}
\put(100,-25){$\circlearrowleft$}

\put(-10,20){$\Phi$}
\put(112,20){$\Phi_s$}
\put(-10,-25){$f_{\phi}$}
\put(112,-25){$f_{\phi_s}$}

\put(50, -18){$\pi_s$}
\put(15, -25){$X_{\phi}$}
\put(80, -25){$X_{\phi_s}$}

\put(15, 20){$\Omega_{\phi}$}
\put(80, 20){$\Omega_{\phi_s}$}
\put(22,12){\vector(0,-1){22}}
\put(85,12){\vector(0,-1){22}}

\put(30,-23){\vector(1,0){43}}
\put(30,22){\vector(1,0){43}}

\end{picture}
\end{equation}

Now if $T\approx_s T'$ (in $\Omega_{\phi}$), then $T-t\approx_s T'-t$ for all $t\in\R$, so $\pi_1(T-t)\bar{R}\pi_1(T'-t)$, i.e., $\pi_s(\pi_1(T-t))=\pi_s(\pi_1(T'-t))$, and hence $\pi_2(\bar{\alpha}(T)-t)=\pi_2(\bar{\alpha}(T')-t)$, for all $t\in\R$. This means that $\bar{\alpha}(T)=\bar{\alpha}(T')$.

Conversely, suppose that $\bar{\alpha}(T)=\bar{\alpha}(T')$. Then $\bar{\alpha}(T-t)=\bar{\alpha}(T'-t)$ for all $t\in\R$ and from diagram \ref{diagram} we have $\pi_1(T-t)\bar{R}\pi_1(T'-t)$ for all $t\in\R$. The set $U$ of $t\in\R$ so that neither $\pi_1(T-t)$ nor $\pi_1(T'-t)$ equals the branch point $E$ of $X_{\phi}$ is open and dense in $\R$. Fix $t\in U$, let $[(x,i)]=\pi_1(T-t)$, and let $[(x',j)]=\pi_1(T'-t)$. Then $(x,i)\bar{R}(x',j)$ and, from Lemma \ref{epsilon}, there is $\epsilon>0$ so that the patches $\{\rho_1-\omega_1-x,\rho_i-x\}$ and $\rho_1-\omega_1-x',\rho_j-x'\}$ are densely eventually coincident at $s$ for a dense set of $s$ in $[-\epsilon,0]$. This means that $T-t$ and $T'-t$ are densely eventually coincident at a dense set of $s$ in $[-\epsilon',0]$ where $\epsilon'>0$ is such that $[-\epsilon,0]\subset spt(\rho_i-x)\cap spt(\rho_j-x')$. It follows that $T$ and $T'$ are densely eventually coincident. It remains to show that $T$ and $T'$ are strongly regionally proximal. 

From Lemma \ref{epsilon} and Corollary \ref{corollary} it follows that there are, up to translation, only finitely many pairs $\{\rho_i-x,\rho_j-x'\}$ with $(x,i)\bar{R}(x',j)$. We are assuming that  $\bar{\alpha}(T)=\bar{\alpha}(T')$. From diagram \ref{diagram},  $\bar{\alpha}(\Phi^{-k}(T))=\bar{\alpha}(\Phi^{-k}(T'))$ for all $k\in\N$. There are then $i,j\in\mathcal{A},x_n\ge0,x'_n\ge0$, and $k_n\to\infty$ so that $\rho_i-x_n\in B_0[\Phi^{-k_n}(T)]$, $\rho_j-x'_n\in B_0[\Phi^{-k_n}(T')]$, and $x_n-x'_n$ is constant. Since $\Phi^{-k_1}(T)$ and $\Phi^{-k_1}(T')$ are densely eventually coincident, there is $s\in spt(\rho_i-x_1)\cap spt(\rho_j-x'_1)$, $N\in\N$, and a tile $\eta$ so that $\eta\in \Phi^N(\rho_i-x_1)\cap\Phi^N(\rho_j-x'_1)$. Then, for $k_n\ge N$, $\Phi^{k_n-N}(\eta)+\lambda^{k_n-N}(x_N-x_1)\in T\cap T'$. Thus $T$ and $T'$ are strongly proximal, and hence strongly regionally proximal. We have shown that $\bar{\alpha}(T)=\bar{\alpha}(T')\Leftrightarrow T\approx_s T'$. Since the $\R$-action on $\Omega_{\phi_s}$ is minimal and $\bar{\alpha}(T-t)=\bar{\alpha}(T)-t$, $\bar{\alpha}$ is surjective, and hence induces an isomorphism of $(\Omega_{\phi}/\approx_s,\R)$ with $(\Omega_{\phi_s},\R)$.

To see that $\approx_s$ is trivial on $(\Omega_{\phi_s},\R)$, suppose that $\bar{\alpha}(T)\approx_s\bar{\alpha}(T')$. In particular, there is a dense set $U$ so that for each $t\in U$ there is $n(t)\in\N$ with 
$B_0[\Phi_s^{n(t)}(\bar{\alpha}(T)-t)]=B_0[\Phi_s^{n(t)}(\bar{\alpha}(T')-t)]$. Fix such a $t$. Then $\pi_2(\Phi_s^{n(t)}(\bar{\alpha}(T)-t))=\pi_2(\Phi_s^{n(t)}(\bar{\alpha}(T')-t))$, so $(\Phi^{n(t)}(T)-\lambda^{n(t)}t)\bar{R}(\Phi^{n(t)}(T')-\lambda^{n(t)}t)$, by diagram \ref{diagram}. Using Lemma \ref{epsilon} (as above), there is then $\epsilon'>0$ so that $\Phi^{n(t)}(T)-\lambda^{n(t)}t$ and $\Phi^{n(t)}(T')-\lambda^{n(t)}t$ are eventually coincident at a dense set of points in $[-\epsilon',0]$. Then $T-t$ and $T'-t$ are eventually coincident at a dense set of points in $[-\epsilon'\lambda^{-n(t)},0]$ and, since the set of such $t$ is dense, $T$ and $T'$ are densely eventually coincident.

We now argue that if $\bar{\alpha}(T)\approx_s\bar{\alpha}(T')$, then $T$ and $T'$ are strongly regionally proximal (denoted $T\sim_{srp}T'$). Let $\pi_{max}:\Omega_{\phi}\to\Omega_{max}$ and $\pi_{max}^s:\Omega_{\phi_s}\to\Omega^s_{max}$ denote the maximal equicontinuous factor maps for the systems $(\Omega_{\phi},\R)$ and $(\Omega_{\phi_s},\R)$, resp., and recall that tilings $T,T'$ in either space have the same image under the maximal equicontinuous factor map if and only if $T\sim_{srp}T'$. Since $\bar{\alpha}(T)=\bar{\alpha}(T')$ implies that $T$ and $T'$ are strongly regionally proximal, $\pi_{max}$ factors through $\Omega_{\phi_s}$. Hence, by maximality, and by uniqueness up to isomorphism of the maximal equicontinuous factor, $(\Omega_{max},\R)$ and $(\Omega^s_{max},\R)$ are isomorphic and we may take $\Omega^s_{max}=\Omega_{max}$ and $\pi_{max}=\pi_{max}^s\circ\bar{\alpha}$. Now $\bar{\alpha}(T)\approx_s\bar{\alpha}(T')\implies \bar{\alpha}(T)\sim_{srp}\bar{\alpha}(T')\implies \pi_{max}^s(\bar{\alpha}(T))=\pi_{max}^s(\bar{\alpha}(T'))\implies T\sim_{srp} T'$. Together with $\bar{\alpha}(T)\approx_s\bar{\alpha}(T')$ implies $T$ and $T'$ are densely eventually coincident, this shows that $\bar{\alpha}(T)\approx_s\bar{\alpha}(T')\implies T\approx_s T'\implies
\bar{\alpha}(T)=\bar{\alpha}(T')$; i.e., $\approx_s$ is trivial on $\Omega_{\phi_s}$.

 For each $z\in\Omega_{max}=\Omega^s_{max}$ there are $T,T'\in\pi_{max}^{-1}(z)$ with $T$ and $T'$ nowhere eventually coincident (\cite{BK}). Then $\bar{\alpha}(T)\ne\bar{\alpha}(T')$ but $\pi_{max}^s(\bar{\alpha}(T))=\pi_{max}^s(\bar{\alpha}(T'))$. Thus, $\pi_{max}^s$ is at least 2-1 everywhere and $(\Omega_{\phi_s},\R)$ does not have pure discrete spectrum.
\end{proof}

\begin{lemma} \label{dec prototiles} Suppose that $\phi$ is primitive, non-periodic, Pisot, injective on initial letters, and constant on final letters. Then there are $i\ne j\in\mathcal{A}$ so that the prototiles $\rho_i$ and $\rho_j$ are eventually coincident at a dense set of points in $spt(\rho_i)\cap spt(\rho_j)$.
\end{lemma}
\begin{proof} By replacing $\phi$ by $\phi^k$ for an appropriate $k\in\N$, we may assume that $\phi(a)=a\cdots $ for all $a\in\mathcal{A}$. Let $r$ be the coincidence rank of $\phi$. There are then $T_1,\ldots,T_r\in\Omega_{\phi}$ with the properties: $T_i\sim_{srp}T_j$ for all $i,j\in\{1,\ldots,r\}$; $\Phi^n(T_i)\cap\Phi^n(T_j)=\emptyset$ for all $i\ne j\in\{1,\ldots,r\}$ and all $n\in\N$; and the $T_i$ are $\Phi$-periodic (see \cite{BK}, or Theorem 4 of \cite{B}). Replacing $\phi$ by an appropriate power, we may assume that the $T_i$ are all fixed by $\Phi$. It follows from the constancy of $\phi$ on final letters that, for $i\ne j\in\{1,\ldots,r\}$, the set of endpoints of $T_i$ (that is, the set of endpoints of supports of tiles in $T_i$) is disjoint from the set of endpoints of $T_j$. Let $V=\{\cdots <v_{-1}<v_0<v_1<\cdots\}$ denote the union of the sets of endpoints of the $T_i$, $i=1,\ldots,r$. For each $k\in\Z$ and $i\in\{1,\ldots,r\}$ let $\tau_i^k$ denote the tile in $T_i$ that has $(v_k+v_{k+1})/2$ in its support. We will call $\mathcal{C}_k:=\{\tau_1^k,\ldots,\tau_r^k\}$ a {\em configuration}. Up to translation, there are only finitely many distinct configurations (see \cite{BK}), and since the $T_i$ are not (translation) periodic, the translation equivalence of $\mathcal{C}_k$ is not a periodic function of $k$. Therefore, there are $k,l\in\Z$ so that $\mathcal{C}_{k-1}-v_{v_k-1}=\mathcal{C}_{l-1}-v_{l-1}$ while $\mathcal{C}_{k}-v_k\ne\mathcal{C}_l-v_l$. For any $n\in\Z$, the configurations $\mathcal{C}_{n-1}$ and $\mathcal{C}_n$ differ only in the tile of $\mathcal{C}_{n-1}$ with terminal endpoint $v_n$ and the tile of $\mathcal{C}_n$ with initial endpoint $v_n$. Hence there are $i\ne j$ with $\rho_i\in \mathcal{C}_k-v_k$, $\rho_j\in\mathcal{C}_l-v_l$, and $(\mathcal{C}_k-v_k)\setminus\{\rho_i\}=(\mathcal{C}_l-v_l)\setminus\{\rho_j\}$.

Suppose (by way of contradiction) that $\rho_i$ and $\rho_j$ are not eventually coincident at a dense set of points in the intersection of their supports.  Note that because $\phi(i)=i\cdots$ and $\phi(j)=j\cdots$, eventual coincidence of $\rho_i$ and $\rho_j$ at a dense set of points in an interval $[0,\delta]$, with $\delta>0$, would imply eventual coincidence of those prototiles at a dense set of points in the intersection of their supports. Thus there are $0\le t_0<t_2\le min\{v_{k+1}-v_k,v_{l+1}-v_l\}$ so that $\rho_i$ and $\rho_j$ are not eventually coincident at any point of the interval $[t_0,t_2]$. Let $i',j'$ be such that $\rho_i\in T_{i'}-v_k$ and $\rho_j\in T_{j'}-v_l$ and let $t_1:=\frac{t_0+t_2}{2}$. By compactness, there are $n_p\to\infty$ so that $\Phi^{n_p}(T_q-v_k-t_1)\to S_q\in\Omega_{\phi}$, for all $q\in\{1,\ldots,r\}$, and $\Phi^{n_p}(T_{j'}-v_l-t_1)\to S_{r+1}\in\Omega_{\phi}$. We claim that the $S_q$, $q=1,\ldots,r+1$, are pairwise strongly regionally proximal. Indeed, strong regional proximality is closed and $\Phi$-invariant, so the $S_q$, $q=1,\ldots,r$, are strongly regionally proximal. Pick $q'\in\{1,\ldots,r\}\setminus\{j_1\}$. There is then $q''\in\{1,\ldots,r\}$ with $B_0[T_{q'}-v_l-t_1]=B_0[T_{q''}-v_k-t_1]$. Then $\Phi^{n_p}(T_{q'}-v_l-t_1)\to S_{q''}$, and again, since $\sim_{srp}$ is closed and $\Phi$-invariant, $S_{r+1}$ is strongly regionally proximal with $S_{q''}$. Hence the $S_q$, $q=1,\ldots,r+1$, are all strongly regionlly proximal, as claimed.  We further claim that the $S_q$, $q\in\{1,\ldots,r+1\}$, are pairwise disjoint. For $q\in\{1,\ldots,r\}$ this is clear: since the $T_q$ are fixed by $\Phi$, the set of configurations formed by the $\Phi^{n_p}(T_q-v_k-t_1)$ equals the set of configurations formed by the $T_q$ and hence the set of configurations formed by the $S_q$ is contained in the set of configurations formed by the $T_q$. In particular, each configuration formed by the $S_q$, $q\in\{1,\ldots,r\}$, consists of $r$ distinct tiles, so these $S_q$ are pairwise disjoint. The same argument shows that $S_{r+1}$ is disjoint from $S_q$ for $q\in\{1,\ldots,r\}\setminus\{i_1\}$, since $\{S_q:q\in\{1,\ldots,r+1\}\setminus\{i_1\}\}=\{\lim_{p\to\infty}\Phi^{n_p}(T_q-v_l-t_1):q\in\{1,\ldots,r\}\}$. It remains to show that $S_{r+1}\cap S_{i_1}=\emptyset$. For this, consider the tilings $U_i,U_j\in\Omega_{\phi}$ that are fixed by $\Phi$ with $\rho_c-\omega_c,\rho_i\in U_i$ and $\rho_c-\omega_c,\rho_j\in U_j$, where $c$ is such that $\phi(a)=\cdots c$ for all $a\in\mathcal{A}$. Then  $S_{i_1}=\lim_{p\to\infty}\Phi^{n_p}(U_i-t_1)$, and  $S_{r+1}=\lim_{p\to\infty}\Phi^{n_p}(U_j-t_1)$. As $U_i\sim_{srp}U_j$ (since $U_i$ and $U_j$ are proximal), there are, up to translation, only finitely many pairs $\{B_0[U_i-t], B_0[U_j-t]\}$, $t\in\R$.
It follows that if there are $t_p$ with $B_0[\lim_{p\to\infty}( U_i-t_p)]=B_0[\lim_{p\to\infty}(U_j-t_p)]$, then 
$B_0[U_i-t_p]=B_0[U_j-t_p]$ for all large $p$. Thus, if $S_{i_1}\cap S_{r+1}\ne\emptyset$, say $B_0[S_{i_1}-t]=B_0[S_{r+1}-t]$, and $t_p=t_1\lambda^{n_p}-t$, we have  $B_0[U_i-t_p]=B_0[\Phi^{n_p}(U_i-t_1)-t]=B_0[\Phi^{n_p}(U_j-t_1)-t]=B_0[U_j-t_p]$ for all large $p$. For $p$ large enough that $|t/\lambda^{n_p}|<(t_0+t_2)/2$, we have $s:=t_1-(t/\lambda^{n_p})\in (t_0,t_2)$ with $\rho_i$ and $\rho_j$ eventually coincident at $s$, a contradiction. This proves that $S_{i_1}$ and $S_{r+1}$ are disjoint and hence the claim that the $S_q$, $q\in\{1,\ldots,r+1\}$, are pairwise disjoint.

We now have the situation that there are $r+1$ tilings $S_1,\ldots,S_{r+1}$ that are strongly regionally proximal and pairwise disjoint. Suppose that $T'=T'_1\in\Omega_{\phi}$. By minimality of the $\R$-action, there are $s_n$ such that $S_1-s_n\to T'_1$, by passing to a subsequence, we may assume that $S_q-s_n\to T'_q\in\Omega_{\phi}$ for $q=1,\ldots,r+1$. Then the $T'_q$ are strongly regionally proximal and pairwise disjoint (the latter uses again that, up to translation, there are only finitely many pairs $\{B_0[S_q-t], B_0[S_{q'}-t]\}$). But then the coincidence rank of $\phi$ is at least $r+1$, not $r$. Thus $\rho_i$ and $\rho_j$ must be eventually coincident at a dense set of points in the intersection of their supports.
\end{proof}

\begin{corollary}\label{nontrivial}  Suppose that $\phi$ is primitive, non-periodic, Pisot, injective on initial letters, and constant on final letters. Then the relation $\approx_s$ is nontrivial on $\Omega_{\phi}$.
\end{corollary}
\begin{proof} Say $\phi(a)=\cdots c$ for all $a\in\mathcal{A}$. Let $\rho_i,\rho_j$ be as in Lemma \ref{dec prototiles} and let $U_i,U_j\in\Omega_{\phi}$ be the $\Phi$-periodic tilings with $\rho_c-\omega_c,\rho_i\in U_i$ and $\rho_c-\omega_c,\rho_j\in U_j$. Then $U_i\ne U_j$ and $U_i\approx_s U_j$.
\end{proof}

\begin{Theorem}\label{main theorem} Suppose that $\phi$ is primitive, Pisot, injective on initial letters, and constant on final letters. Then $(\Omega_{\phi},\R)$ has pure discrete spectrum.
\end{Theorem}

\begin{proof} Suppose that $(\Omega_{\phi},\R)$ does not have pure discrete spectrum. Then $\phi_s$ is also primitive, Pisot, injective on initial letters, and constant on final letters by Lemma \ref{shape of phi sub s}. By Theorem \ref{structure}, $\approx_s$ is trivial on $\Omega_{\phi_s}$.  If there were a translation periodic tiling in $\Omega_{\phi_s}$, it would follow from primitivity of $\phi_s$ that $(\Omega_{\phi_s},\R)$, being simply translation on a circle, has pure discrete spectrum, which it doesn't by Theorem \ref{structure}. Thus $\phi_s$ is non-periodic. But now Corollary \ref{nontrivial} says $\approx_s$ is nontrivial on $\Omega_{\phi_s}$. Thus $(\Omega_{\phi},\R)$ must have pure discrete spectrum.
\end{proof}

\begin{Remark}\label{Remark} By replacing $\phi$ by $\phi^n$ for appropriate $n\in\N$, the hypothesis in Theorem \ref{main theorem}
that $\phi$ is constant on final letters can be weakened to $\phi$ being eventually constant on final letters. Also, if $\phi$ satisfies the other hypotheses but is not necessarily primitive, let $c$ be the (eventual) last letter of all $\phi(a)$ and let $\mathcal{A}':=\{a\in\mathcal{A}:\exists n\in\N\text{ with }a \text{ occurring in }\phi^n(c)\}$. Then the theorem applies to $\phi':=\phi|_{\mathcal{A}'}$ to conclude that the `minimal core' $(\Omega_{\phi'},\R)$ of $(\Omega_{\phi},\R)$ has pure discrete spectrum.
\end{Remark}

\section{Applications}\label{applications}
\subsection{$\beta$-substitutions}
Given $\beta>1$, let $T_{\beta}:[0,1]\to[0,1]$ be the {\em $\beta$-transformation} defined by $x\mapsto \beta x-\lfloor\beta x\rfloor$.  The number $\beta$ is called a {\em Parry number} if the orbit of 1 under $T_{\beta}$ is finite and a {\em simple Parry number} if $T_{\beta}^n(1)=0$ for some $n\in\N$. All Pisot numbers are Parry numbers (\cite{Ber, Sch}). 

Given a simple Parry number $\beta$ with $n$ (smallest) as above, let $P_i:=[0,T_{\beta}^{i-1}(1)]$ for $i=1,\ldots,n$. Then, for $i=1,\ldots,n-1$, $T_{\beta}$ maps $P_i$ $a_i$ times across $P_1=[0,1]$ and once across $P_{i+1}$, and $T_{\beta}$ maps $P_n$ exactly $a_n$ times across $P_1$. The {\em $\beta$-substitution}, $\phi_{\beta}$, is given by:

\begin{eqnarray*}
\phi_{\beta}(1) &=& 21^{a_1} \\
\phi_{\beta}(2) &=& 31^{a_2} \\
&\vdots&\\
\phi_{\beta}(n) &=& 1^{a_n}.
\end{eqnarray*}

Note that $a_1\ne0\ne a_n$ so that $\phi_{\beta}^n(i)=\cdots 1$ for all $i\in\{1,\ldots,n\}$; hence $\phi_{\beta}$ is eventually constant on final letters.

Taking $\rho_i:=(P_i,i)$, $i=1,\ldots,n$, as prototiles for $\phi_{\beta}$, we define (a.e.) the map $p:\Omega_{\phi_{\beta}}\to \inv T_{\beta}$ by $p(T):=(\ldots,t_{-1},t_0,t_1,\ldots)$ with $t_k$ determined by: $\rho_i-t_k\in\Phi_{\beta}^{-k}(T)$ and $t_k\in int(spt(\rho_i))$ for some $i$. Then $p$ is a metric isomorphism that conjugates the substitution-induced homeomorphism on $\Omega_{\phi_{\beta}}$ with the shift homeomorphism $\hat{T}_{\beta}$ on the inverse limit space $\inv T_{\beta}$. 

Each non-negative real number $x$ has a {\em greedy expansion} in base $\beta$: $$x=\sum_{k=-N}^{\infty}x_k\beta^{-k}$$
with $x_k\in\{0,\ldots,\lfloor\beta\rfloor\}$ satisfying $|x_M-\sum_{k=-N}^Mx_k\beta^{-k}|<\beta^{-M}$ for each $M\ge -N$.
Each such greedy expansion determines $(\ldots,0,0,x_{-N},\ldots,x_{-1},x_0,x_1,\ldots)\in\{0,\ldots,\lfloor\beta\rfloor\}^{\Z}$; the closure of all such sequences is denoted by $\Sigma_{\beta}$ and called the {\em $\beta$-shift}.\footnote{Not to be confused with the substitutive system $\Sigma_{\phi_{\beta}}$.} For $\beta$ a simple Parry number, $\Sigma_{\beta}$ is a subshift of finite type. The map $r:(\ldots,x_{-1},x_0,x_1,\ldots)\mapsto (\ldots,\sum_{k=1}^{\infty}x_{k-1}\beta^{-k}, \sum_{k=1}^{\infty}x_{k}\beta^{-k},\sum_{k=1}^{\infty}x_{k+1}\beta^{-k},\ldots)$ defines a metric isomorphism that conjugates the shift $\sigma$ on $\Sigma_{\beta}$ with $\hat{T}_{\beta}$ on $\inv_{T_{\beta}}$. The map $$g:=\pi_{max}\circ p^{-1}\circ r:\Sigma_{\beta}\to\Omega_{max},$$ where $\pi_{max}:\Omega_{\phi_{\beta}}\to\Omega_{max}$ is the maximal equicontinuous factor map, gives a continuous and bounded-to-one semi-conjugacy between the $\beta$-shift $\sigma$
and the hyperbolic automorphism $(\Phi_{\beta})_{max}:\Omega_{max}\to\Omega_{max}$. This map $g$ has the nice property of being an {\em arithmetical coding}\footnote{The terminology is due to Siderov (\cite{Sid}).}: If $x$ and $x'$ are non-negative real numbers with sequences $\underline{x},\underline{x'}$ and $\underline{x+x'}$ in $\Sigma_{\beta}$ corresponding to the greedy $\beta$-expansions of $x,x'$ and $x+x'$, then $g(\underline{x+x'})=g(\underline{x})+g(\underline{x'})$. Since $\phi_{\beta}$ is primitive, injective on initial letters, and eventually constant on final letters (see Remark \ref{Remark}), it follows from Theorem \ref{main theorem} that, for Pisot simple Parry numbers, the map $\pi_{max}$, and hence the coding $g$, is a.e. one-to-one. (This result appears in \cite{BBK} under the additional hypothesis that the algebraic degree of $\beta$ is greater than $n/p$, $p$ the smallest prime divisor of $n$.)

\begin{corollary} If $\beta$ is a Pisot simple Parry number, then the system $(\Omega_{\phi_{\beta}},\R)$ has pure discrete spectrum.
\end{corollary} 

\begin{Remark}\label{Remark2} For $\beta$ a Pisot simple Parry number and a unit, the substitutive system $(\Sigma_{\phi_{\beta}},\Z)$ also has pure discrete spectrum provided the substitution $\phi_{\beta}$ is irreducible (that is, $n=d$, $d$ the algebraic degree of $\beta$ - \cite{BKw} or \cite{CS}). But if $\phi_{\beta}$ is reducible this is not necessarily the case as counterexamples of Ei and Ito show (\cite{EI}). In any event, if $\beta$ is a unit, $(\Sigma_{\phi_{\beta}},\Z)$ is measurably conjugate to the system of a map induced by a rotation on the $(d-1)$-torus (\cite{BBK}, Proposition 8.1).
\end{Remark}

\subsection{Arnoux-Rauzy, Brun, and Jacobi-Perron substitutions}

The substitutions of this section arise as finite products of elements taken from certain collections of basic substitutions. The Arnoux-Rauzy substitutions generalize the two-letter Sturmian substitutions and the three-letter Rauzy substitution (see \cite{BFZ} and \cite{CC} for general accounts of these) and the Brun and Jacobi-Perron substitutions come from multidimensional continued fraction algorithms (see \cite{Be} and \cite{S}).

 Let $\mathcal{A}=\{1,\ldots,d\}$ and for each $i\in\mathcal{A}$, let $\sigma_i$ be the substitution on $\mathcal{A}$ defined by $\sigma_i(i)=i$ and $\sigma_i(j)=ji$ for $j\ne i$. Given a word $w=w_1\cdots w_k\in\mathcal{A}^*$ that contains at least one occurrence of each letter of $\mathcal{A}$, let $\sigma_w=\sigma_{w_1}\circ\cdots\circ\sigma_{w_k}$. Such a $\sigma_w$ is called an {\em Arnoux-Rauzy substitution}. Arnoux and Ito (\cite{AI}) first proved that all Arnoux-Rauzy substitutions on 2 or 3 letters are irreducible Pisot and recently Avila and Delecroix have proved that  Arnoux-Rauzy substitutions on any number of letters are irreducible Pisot (\cite{AD}). The following (assuming Pisot) is proved in \cite{BSW} by analyzing balanced pairs and a rather different proof, for $d=3$,  has been given by Berth\'{e}, Jolivet, and Siegel (\cite{BJS}). It is easily verified that Arnoux-Rauzy substitutions are primitive, injective on initial letters, and 
constant on final letters.

\begin{corollary} \label{AR} If $\sigma_w$ is an Arnoux-Rauzy substitution, then $(\Omega_{\sigma_w},\R)$ has pure discrete spectrum.
\end{corollary} 

The {\em Brun substitutions}
$$\sigma_1^{Brun}:1\mapsto 1,\, 2\mapsto 2,\, 3\mapsto 32;$$
$$\sigma_2^{Brun}:1\mapsto 1,\, 2\mapsto 3,\, 3\mapsto 23; \text{ and }$$
$$\sigma_3^{Brun}:1\mapsto 2,\, 2\mapsto 3,\, 3\mapsto 13$$
are discussed  by Berth\'{e}, Bourdon, Jolivet, and Siegel in \cite{BBJS} where they prove that if $w=w_1\cdots w_k\in\{1,2,3\}^*$ is any word containing at least one occurrence of 3 and $\phi=\sigma_{w_1}^{Brun}\circ\cdots\circ\sigma_{w_k}^{Brun}$, then $(\Omega_{\phi},\R)$ has pure discrete spectrum.
\\
The {\em Jacobi-Perron substitutions} are
$$\sigma_{a,b}^{JP}:1\mapsto 3,\,2\mapsto 13^a,\, 3\mapsto23^b$$
for $a,b\in\N_0$. Any product $\phi=\sigma_{a_1,b_1}^{JP}\circ\cdots\circ\sigma_{a_n,b_n}^{JP}$ is irreducible Pisot as long as $0\le a_i\le b_i$ and $b_i\ne0$ for $i=1,\ldots,n$ (\cite{DFP}). It is proved in \cite{BBJS} that, for such $\phi$, $(\Omega_{\phi},\R)$ has pure discrete spectrum.

The results for the Brun and Jacobi-Perron substitutions follow immediately from Theorem \ref{main theorem} and Remark \ref{Remark}.


\begin{thebibliography}{99}

\bibitem[ABBLS]{ABBLS} S. Akiyama, M. Barge, V. Berth\'{e}, J.-Y. Lee, A. Siegel, On the Pisot Conjecture, in ``Mathematics of Aperiodic Order,'' {\em Progress in Mathematics} {\bf 309} (2015), 33-72.

\bibitem[AL1]{AL1} S. Akiyama and J.-Y. Lee, Algorithm for determining pure pointedness of
 self-affine tilings, {\em Adv. Math.} \textbf{226}, no. 4, (2011), 2855-2883.

\bibitem[AL2]{AL2} S. Akiyama and J.-Y. Lee, Computation of pure discrete spectrum of self-affine tilings. Preprint, 2013.

\bibitem[AP]{AP}J.E. Anderson and I.F. Putnam, Topological invariants
  for substitution tilings and their associated $C^*$-algebras, {\em
    Ergodic Theory \& Dynamical Systems} \textbf{18} (1998), 509--537.

\bibitem[AI]{AI} P. Arnoux and S. Ito, 
   Pisot Substitutions and Rauzy fractals,
   {\em Bull. Belg. Math Soc.} \textbf{8} (2001), 181-207.
       

    
\bibitem[AD]{AD} A. Avila and V. Delecroix, Some monoids of Pisot matrices, arXiv: 1506.03692v1.


  
\bibitem[BBK]{BBK} V. Baker, M. Barge and J. Kwapisz, Geometric realization and coincidence for reducible non-unimodular Pisot tiling spaces with an application to $\beta$-shifts,
 {\em Ann. Inst. Fourier (Grenoble)} {\bf 56} No. 7 (2006), p. 2213--2248.
 
\bibitem[B]{B} M. Barge, Factors of Pisot tiling spaces and the Coincidence Rank Conjecture, 	arXiv:1301.7094, (2013)
  to appear in {\em Bull. Soc. Math. France}.

\bibitem[BD2]{BD2} M.\ Barge, B.\ Diamond, A complete invariant for the topology of one- 
  dimensional substitution tiling spaces,
  {\em Ergod. Th. \& Dyn. Sys.} {\bf 21} (2001), 1333-1358.


\bibitem[BD1]{BD} M. Barge and B. Diamond,
Coincidence for substitutions of Pisot type,
{\em Bull. Soc. Math. France} {\bf 130} (2002), p. 691--626.    

\bibitem[BBJS]{BBJS} V. Berth\'{e}, J. Bourdon, T. Jolivet, and A. Siegel, A combinatorial approach to products of Pisot substitutions, arXiv:1401.0704v1, (2014).

\bibitem[BK]{BK} M. Barge and J. Kellendonk, Proximality and pure point spectrum for  
  tiling dynamical systems, {\em Michigan Math. J.} {\bf 62}, no. 4, (2013), 793-822.

\bibitem[BKw]{BKw} M. Barge and J. Kwapisz, Geometric theory of
  unimodular Pisot substitutions, \textit{Amer J. Math.} {\bf128}
  (2006), 1219-1282.
  
  
\bibitem[BSW]{BSW} M. Barge, S. \v{S}timac and R.F. Williams, Pure discrete spectrum in  
  substitution tiling spaces, {\em Disc. and Cont. Dynam. Sys. - A}{\bf 2} (2013), 579-597.


\bibitem[BFZ]{BFZ} V. Berth\'{e}, S. Ferenczi, and L. Q. Zamboni, Interactions between dynamics, arihtmetics and combinatorics: the good, the bad, and the ugly, {\em Algebraic and topological dynamics, Contemp. Math.}, {\bf 385}, Amer. Math. Soc., Providence, RI (2005), 333-364.

\bibitem[BJS]{BJS} V. Berth\'{e}, T. Jolivet and A. Siegel, 
   Substitutive Arnoux-Rauzy substitutions  
  have pure discrete spectrum, {\em Unif. Distrib. Theory} {\bf 7}, no. 1, (2012), 173-197.
  
\bibitem[Be]{Be} V. Berth\'{e}, Multidimensional Euclidean algorithms, numeration and substitutions, {\em Integers} {\bf 11B} (2011), Paper No. A2, 34.
  
\bibitem[Ber]{Ber} A. Bertrand, D\'{e}veloppements en base de Pisot et r\'{e}partition modulo 1, {\em C. R. Acad.
Sci. Paris}, {\bf 285}, no.6, (1977), A419-A421.

\bibitem[CC]{CC} J. Cassaigne and N. Chekhova, Fonctions de r\'{e}currence des suites d'Arnoux-Rauzy et r\'{e}ponse \`{a} une questionde Morseet Hedlund, {\em Ann. Inst. Fourier (Grenoble)} {\bf 56}, no. 7, Num\'{e}ration, pavages, substitutions (2006), 2249-2270.

\bibitem[CS]{CS} A. Clark and L. Sadun, When size matters: subshifts and their related tiling spaces,
{\em Ergodic Theory Dynam. Systems}, {\bf23} (2003), 1043-1057.

\bibitem[DFP]{DFP} E. Dubois, A. Farhane, and R. Paysant-Le Roux, The Jacobi-Perron algorithm and Pisot numbers, {\em Acta Arith.} {\bf 111} (2004), 269-275.

\bibitem[EI]{EI} H. Ei and S. Ito, Tilings from some non-irreducible, Pisot substitutions, {\em Discrete
Math. and Theo. Comp. Science} {\bf 8}, no.1 (2005), 81-122.


\bibitem[HS]{HS} M. Hollander and B. Solomyak, Two-symbol Pisot substitutions have pure discrete spectrum, {\em  Ergodic Theory \& Dynamical Systems} {\bf 23} (2003), 533-540.

\bibitem[M]{M} B. Moss\'{e}, Puissances de mots et reconnaissabilit\'{e} des points fixes d'une substitution,
{\em Theoretical Computer Science} {\bf99} (1992), 327-334.


\bibitem[S]{S} F. Schweiger, {\em Multidimensional continued fractions}, Oxford Science Publications, Oxford University Press, Oxford, 2000.

\bibitem[Sc]{Sch} K. Schmidt, On periodic expansions of Pisot numbers and Salem numbers, {\em Bull.
London Math. Soc.}, {\bf 12} (1980), 269-278.

\bibitem[Si]{Sid} N. Sidorov, Arithmetic dynamics, in {\em Topics in dynamics and ergodic theory},
vol. 310 of {\em  London Math. Soc. Lecture Note Ser.}, pp. 145-189, Cambridge:
Cambridge Univ. Press, 2003.

 
\bibitem[S1]{S1} B. Solomyak, Dynamics of self-similar tilings, {\em  Ergod. Th. \& Dynam. Sys.},
{\bf 17} (1997), 695-738.
   
\bibitem[S2]{Sol} B. Solomyak, Nonperiodicity implies unique
  composition for self-similar translationally finite tilings, {\em
    Discrete Comput. Geometry} {\bf 20} (1998), 265--279.
    
     
\bibitem[W]{W} P. Walters, {\em An introduction to ergodic theory}, Springer-Verlag, New York, 1982.
     
\end{thebibliography}
\end{document}